\documentclass[12pt,abstract]{scrartcl}
\usepackage[dvipdfmx]{graphicx}
\usepackage{amsmath,amsthm}
\makeatletter
\renewcommand{\theenumii}{\@roman\c@enumii}
\makeatother

\usepackage{natbib}
\newtheorem{lemma}{Lemma}
\newtheorem{theorem}{Theorem}
\newtheorem{assumption}{Assumption}

\usepackage{amssymb}
\usepackage{newtxtext}
\usepackage[varg,vvarbb,libertine]{newtxmath}
\usepackage{cabin}
\def\gra{\{1, 2, \dots, m\}}
\def\grb{\{m+1, m+2, \dots, n\}}
\begin{document}

\title{Sion's mini-max theorem and Nash equilibrium in a multi-players game with two groups which is zero-sum and symmetric in each group\thanks{This work was supported by Japan Society for the Promotion of Science KAKENHI Grant Number 15K03481 and 18K01594.}}

\small{
\author{
Atsuhiro Satoh\thanks{atsatoh@hgu.jp}\\[.01cm]
Faculty of Economics, Hokkai-Gakuen University,\\[.02cm]
Toyohira-ku, Sapporo, Hokkaido, 062-8605, Japan,\\[.01cm]
\textrm{and} \\[.1cm]
Yasuhito Tanaka\thanks{yasuhito@mail.doshisha.ac.jp}\\[.01cm]
Faculty of Economics, Doshisha University,\\
Kamigyo-ku, Kyoto, 602-8580, Japan.\\}
}

\date{}

\maketitle
\thispagestyle{empty}

\vspace{-1.25cm}

\begin{abstract}
We consider the relation between Sion's minimax theorem for a continuous function and a Nash equilibrium in a multi-players game with two groups which is zero-sum and symmetric in each group. We will show the following results.
\begin{enumerate}
\item The existence of Nash equilibrium which is symmetric in each group implies Sion's minimax theorem with the coincidence of the maximin strategy and the minimax strategy for players in each group. 
	\item Sion's minimax theorem with the coincidence of the maximin strategy and the minimax strategy for players in each group implies the existence of a Nash equilibrium which is symmetric in each group.
\end{enumerate}
Thus, they are equivalent. An example of such a game is a relative profit maximization game in each group under oligopoly with two groups such that firms in each group have the same cost functions and maximize their relative profits in each group, and the demand functions are symmetric for the firms in each group.
\end{abstract}

\begin{description}
	\item[Keywords:] multi-players zero-sum game, two groups, Nash equilibrium, Sion's minimax theorem
\end{description}

\begin{description}
	\item[JEL Classification:] C72
\end{description}

\section{Introduction}

We consider the relation between Sion's minimax theorem for a continuous function and the existence of Nash equilibrium in a multi-players game with two groups which is zero-sum and symmetric in each group. There are $n$ players. Players $1, 2, \dots, m$ are in one group, and Players $m+1, m+2, \dots, n$ are in the other group. We assume $n\geq 4$ and $2\leq m\leq n-2$. Thus, each group has at least two players. Players $1, 2, \dots, m$ have the same payoff functions and strategy spaces, and they play a game which is zero-sum in this group, that is, the sum of the payoffs of Players $1, 2, \dots, m$ is zero. Similarly, Players $m+1, m+2, \dots, n$ have the same payoff functions and strategy spaces, and they play a game which is zero-sum in this group, that is, the sum of the payoffs of Players $m+1, m+2, \dots, n$ is zero.

We will show the following results.
\begin{enumerate}
	\item The existence of a Nash equilibrium which is symmetric in each group implies Sion's minimax theorem with the coincidence of the maximin strategy and the minimax strategy for players in each group. 
	\item Sion's minimax theorem for players with the coincidence of the maximin strategy and the minimax strategy in each group implies the existence of Nash equilibrium which is symmetric in each group.
\end{enumerate}
Thus, they are equivalent.

An example of such a game is a relative profit maximization game in each group under oligopoly with two groups such that firms in each group have the same cost functions and maximize their relative profits in each group, and demand functions are symmetric for the firms in each group. Assume that there are six firms, A, B, C, D, E and F. Let $\bar{\pi}_A$, $\bar{\pi}_B$, $\bar{\pi}_C$, $\bar{\pi}_D$, $\bar{\pi}_E$ and $\bar{\pi}_F$ be the absolute profits of, respectively, Firms A, B, C, D, E and F. Firms A, B and E have the same cost function, and the demand functions are symmetric for them. Firms C, D and F have the same cost function, and the demand functions are symmetric for them. However, the firms in different groups have different cost functions, and the demand functions are not symmetric for firms in different groups.

The relative profits of Firms A, B and E are
\[\pi_A=\bar{\pi}_A-\frac{1}{2}(\bar{\pi}_B+\bar{\pi}_E),\]
\[\pi_B=\bar{\pi}_B-\frac{1}{2}(\bar{\pi}_A+\bar{\pi}_E),\]
\[\pi_E=\bar{\pi}_E-\frac{1}{2}(\bar{\pi}_A+\bar{\pi}_B).\]
The relative profits of Firms C, D and F are
\[\pi_C=\bar{\pi}_C-\frac{1}{2}(\bar{\pi}_D+\bar{\pi}_F),\]
\[\pi_D=\bar{\pi}_D-\frac{1}{2}(\bar{\pi}_C+\bar{\pi}_F),\]
\[\pi_F=\bar{\pi}_F-\frac{1}{2}(\bar{\pi}_C+\bar{\pi}_D).\]
We see
\[\pi_A+\pi_B+\pi_E=0,\]
\[\pi_C+\pi_D+\pi_F=0.\]
Firms A, B, C, D, E and F maximize, respectively, $\pi_A$, $\pi_B$, $\pi_C$, $\pi_D$, $\pi_E$ and $\pi_F$. Thus, the relative profit maximization game in each group is a zero-sum game\footnote{About relative profit maximization under imperfect competition please see \cite{mm}, \cite{ebl2}, \cite{eb2}, \cite{st}, \cite{eb1}, \cite{ebl1} and \cite{redondo}}. In Section \ref{ex} we present an example of relative profit maximization in each group under oligopoly with two groups. 

We think that our analysis can be easily extended to a case with more than two groups.

\section{The model and Sion's minimax theorem}

Consider a multi-players game with two groups which is zero-sum and symmetric in each group. There are $n$ players. Players $1, 2, \dots, m$ are in one group, and Players $m+1, m+2, \dots, n$ are in the other group. We assume $n\geq 4$ and $2\leq m\leq n-2$. Thus, each group has at least two players. Players $1, 2, \dots, m$ have the same payoff functions and strategy spaces, and they play a game which is zero-sum in this group, that is, the sum of the payoffs of Players $1, 2, \dots, m$ is zero. Similarly, Players $m+1, m+2, \dots, n$ have the same payoff functions and strategy spaces, and they play a game which is zero-sum in this group, that is, the sum of the payoffs of Players $m+1, m+2, \dots, n$ is zero. The strategic variables for the players are $s_1$, $s_2$, \dots, $s_n$, and ($s_1, s_2, \dots, s_n)\in S_1\times S_2\times \dots \times S_n$. $S_1$, $S_2,\ \dots,\ S_n$ are convex and compact sets in linear topological spaces. 

The payoff function of each player is $u_i(s_1, s_2,\dots, s_n), i=1, 2, \dots, n$. We assume 
\begin{quote}
$u_i$'s for $i=1, 2, \dots, n$ are continuous real-valued functions on $S_1\times S_2\times \dots \times S_n$, quasi-concave on $S_i$ for each $s_j\in S_j,\ j\neq i$, and quasi-convex on $S_j$ for $j\neq i$ for each $s_i\in S_i$.
\end{quote}

Since the game is zero-sum in each group, we have
\begin{equation}
u_1(s_1, s_2,\dots,  s_n)+u_2(s_1, s_2,\dots,  s_n)+\dots, u_m(s_1, s_2,\dots,  s_n)=0,\label{e1}
\end{equation}
\begin{equation}
u_{m+1}(s_1, s_2,\dots,  s_n)+u_{m+2}(s_1, s_2,\dots,  s_n)+\dots, u_n(s_1, s_2,\dots,  s_n)=0,\label{e1a}
\end{equation}
for given $(s_1, s_2,\dots, s_n)$.

Sion's minimax theorem (\cite{sion}, \cite{komiya}, \cite{kind}) for a continuous function is stated as follows.
\begin{lemma}
Let $X$ and $Y$ be non-void convex and compact subsets of two linear topological spaces, and let $f:X\times Y \rightarrow \mathbb{R}$ be a function, that is continuous and quasi-concave in the first variable and continuous and quasi-convex in the second variable. Then
\[\max_{x\in X}\min_{y\in Y}f(x,y)=\min_{y\in Y}\max_{x\in X}f(x,y).\] \label{l1}
\end{lemma}
We follow the description of this theorem in \cite{kind}.

Let $s_h$'s for $h\neq i, j;\ i,j \in \{1, 2, \dots, m\}$ be given; then, $u_i(s_1, s_2, \dots, s_n)$ is a function of $s_i$ and $s_j$. We can apply Lemma \ref{l1} to such a situation, and get the following equation.
\begin{equation}
\max_{s_i\in S_i}\min_{s_j\in S_j}u_i(s_1, s_2, \dots, s_n)=\min_{s_j\in S_j}\max_{s_i\in S_i}u_i(s_1, s_2, \dots, s_n).\label{as0}
\end{equation}
By symmetry
\begin{equation*}
\max_{s_j\in S_j}\min_{s_i\in S_i}u_j(s_1, s_2, \dots, s_n)=\min_{s_i\in S_i}\max_{s_j\in S_j}u_j(s_1, s_2, \dots, s_n).
\end{equation*}
Similarly, let $s_h$'s for $h\neq k, l;\ k,l \in \{m+1, m+2, \dots, n\}$ be given; then we obtain
\begin{equation}
\max_{s_k\in S_k}\min_{s_l\in S_l}u_k(s_1, s_2, \dots, s_n)=\min_{s_l\in S_l}\max_{s_k\in S_k}u_i(s_1, s_2, \dots, s_n).\label{as0a}
\end{equation}
By symmetry
\begin{equation*}
\max_{s_l\in S_l}\min_{s_k\in S_k}u_l(s_1, s_2, \dots, s_n)=\min_{s_k\in S_k}\max_{s_l\in S_l}u_j(s_1, s_2, \dots, s_n).
\end{equation*}

We assume that $\arg\max_{s_i\in S_i}\min_{s_j\in S_j}u_i(s_1, s_2, \dots, s_n)$, $\arg\min_{s_j\in S_j}\max_{s_i\in S_i}u_i(s_1, s_2, \dots, s_n)$ and so on are unique, that is, single-valued. By the maximum theorem they are continuous in $s_h$'s,$\ h\neq i, j$ or in $s_h$'s,$\ h\neq k, l$. Also, throughout this paper we assume that the maximin strategy and the minimax strategy of players in any situation are unique, and the best responses of players in any situation are unique.

Let us consider a point such that $s_i=s$ for $i\in \{1, 2, \dots, m\}$ and $s_k=s'$ for $k\in \{m+1, m+2, \dots, n\}$, and consider the following function.
\[
\begin{pmatrix}
s\\
s'
\end{pmatrix}
\rightarrow 
\begin{pmatrix}
\arg\max_{s_i\in S_i}\min_{s_j\in S_j}u_i(s_i,s_j,s,\dots,s,s',\dots,s')\\
\arg\max_{s_k\in S_k}\min_{s_l\in S_l}u_k(s,\dots,s,s_k, s_l,s',\dots,s')
\end{pmatrix},
\]
for $\ i\in \{1, 2, \dots, m\}$, $k\in \{m+1,m+2,\dots,n\}$. Since $u_i$ and $u_k$ are continuous, $S_i=S_j$ is compact and $S_k=S_l$ is compact, these functions are also continuous. Thus, there exists a fixed point of $(s,s')$. Denote it by $(\tilde{s},\hat{s})$. It satisfies
\begin{equation}
\tilde{s}=\arg\max_{s_i\in S_i}\min_{s_j\in S_j}u_i(s_i,s_j,\tilde{s},\dots,\tilde{s},\hat{s},\dots,\hat{s}),\ i\in \{1, 2, \dots, m\},\label{fix}
\end{equation}
\begin{equation}
\hat{s}=\arg\max_{s_k\in S_k}\min_{s_l\in S_l}u_k(\tilde{s},\dots,\tilde{s},s_k,s_l,\hat{s},\dots,\hat{s}),\ k\in \{m+1,m+2,\dots,n\}.\label{fixa}
\end{equation}

Now we assume 
\begin{assumption}
About $\tilde{s}$ and $\hat{s}$ which satisfy (\ref{fix}) and (\ref{fixa}),
\begin{align*}
\arg\max_{s_i\in S_i}\min_{s_j\in S_j}u_i(s_i,s_j,\tilde{s},\dots,\tilde{s},\hat{s},\dots,\hat{s})=\arg\min_{s_j\in S_j}\max_{s_i\in S_i}u_i(s_i,s_j,\tilde{s},\dots,\tilde{s},\hat{s},\dots,\hat{s}), 
\end{align*}
\[\arg\max_{s_k\in S_k}\min_{s_l\in S_l}u_k(\tilde{s},\dots,\tilde{s},s_k,s_l,\hat{s},\dots,\hat{s})=\arg\min_{s_l\in S_l}\max_{s_k\in S_k}u_k(\tilde{s},\dots,\tilde{s},s_k,s_l,\hat{s},\dots,\hat{s}),\]
for $\ i\in \{1, 2, \dots, m\}$, $k\in \{m+1,m+2,\dots,n\}$, that is, the maximin strategy and the minimax strategy coincide.\label{as1}
\end{assumption}

As we will show in the Appendix, without Assumption \ref{as1} we may have a Nash equilibrium which is asymmetric in each group.

\section{The main results}

Consider a Nash equilibrium which is symmetric in each group. Let $s_i^*$'s and $s_k^*$'s be the values of $s_i$'s for $i\in \{1, 2, \dots, m\}$ and $s_k$'s for $k\in \{m+1, m+2, \dots, n\}$ which, respectively, maximize $u_i$'s and $u_k$'s, that is, 
\[u_i(s_1^*,s_2^*,\dots,s_i^*,\dots,s_n^*)\geq u_i(s_1^*,s_2^*,\dots,s_i,\dots,s_n^*)\ \mathrm{for\ any}\ s_i\in S_i,\]
and
\[u_k(s_1^*,s_2^*,\dots,s_k^*,\dots,s_n^*)\geq u_k(s_1^*,s_2^*,\dots,s_k,\dots,s_n^*)\ \mathrm{for\ any}\ s_k\in S_k,\]
If the Nash equilibrium is symmetric in each group, $s_1^*$'s for all $i\in \{1,2, \dots, m\}$ are equal, and $s_k^*$'s for all $k\in \{m+1,m+2, \dots, n\}$ are equal.

Notations of strategy choice by players are as follows.
\begin{quote}
$(s_i,s^*,\dots,s^*,s^{**},\dots,s^{**})$ is a vector of strategy choice by players such that Players 1, $\dots$, $m$ other than $i$ choose $s^*$ and Players $m+1$, $\dots$, $n$ choose $s^{**}$. $(s^*,\dots,s^*,s_k,s^{**},\dots,s^{**})$ is a vector such that Players 1, $\dots$, $m$ choose $s^*$ and Players $m+1$, $\dots$, $n$ other than $k$ choose $s^{**}$. $(s_i,s_j,s^*,\dots,s^*,s^{**},\dots,s^{**})$ is a vector such that Players 1, $\dots$, $m$ other than $i$ and $j$ choose $s^*$ and Players $m+1$, $\dots$, $n$ choose $s^{**}$. $(s^*,\dots,s^*,s_k,s_l,s^{**},\dots,s^{**})$ is a vector such that Players 1, $\dots$, $m$ choose $s^*$ and Players $m+1$, $\dots$, $n$ other than $k$ and $l$ choose $s^{**}$. 

$(s_i,\tilde{s},\dots,\tilde{s},\hat{s},\dots,\hat{s})$ is a vector of strategy choice by players such that Players 1, $\dots$, $m$ other than $i$ choose $\tilde{s}$ and Players $m+1$, $\dots$, $n$ choose $\hat{s}$. $(\tilde{s},\dots,\tilde{s},s_k,\hat{s},\dots,\hat{s})$ is a vector such that Players 1, $\dots$, $m$ choose $\tilde{s}$ and Players $m+1$, $\dots$, $n$ other than $k$ choose $\hat{s}$. $(s_i,s_j,\tilde{s},\dots,\tilde{s},\hat{s},\dots,\hat{s})$ is a vector such that Players 1, $\dots$, $m$ other than $i$ and $j$ choose $\tilde{s}$ and Players $m+1$, $\dots$, $n$ choose $\hat{s}$. $(\tilde{s},\dots,\tilde{s},s_k,s_l,\hat{s},\dots,\hat{s})$ is a vector such that Players 1, $\dots$, $m$ choose $\tilde{s}$ and Players $m+1$, $\dots$, $n$ other than $k$ and $l$ choose $\hat{s}$.

The same applies to other similar notations.
\end{quote}

We show the  following theorem.
\begin{theorem}
The existence of Nash equilibrium which is symmetric in each group implies Sion's minimax theorem with the coincidence of the maximin strategy and the minimax strategy. 
\label{t1}
\end{theorem}
\begin{proof}
Let $(s_1,\dots,s_m, s_{m+1}, \dots, s_n)=(s^*,\dots,s^*,s^{**},\dots,s^{**})$ be a Nash equilibrium which is symmetric in each group. Since the game is zero-sum in each group. 
\[u_i(s_i,s^*,\dots,s^*,s^{**},\dots,s^{**})+\sum_{j=1,j\neq i}^mu_j(s_i,s^*,\dots,s^*,s^{**},\dots,s^{**})=0,\]
and
\[u_k(s^*,\dots,s^*,s_k,s^{**},\dots,s^{**})+\sum_{l=m+1,k\neq k}^nu_l(s^*,\dots,s^*,s_k,s^{**},\dots,s^{**})=0\]
imply
\[u_i(s_i,s^*,\dots,s^*,s^{**},\dots,s^{**})=-(m-1)u_j(s_i,s^*,\dots,s^*,s^{**},\dots,s^{**}),\]
and
\[u_k(s^*,\dots,s^*,s_k,s^{**},\dots,s^{**})=-(n-m-1)u_l(s^*,\dots,s^*,s_k,s^{**},\dots,s^{**}).\]
These equations hold for any $s_i$ and $s_k$. Therefore,
\[\arg\max_{s_i\in S_i}u_i(s_i,s^*,\dots,s^*,s^{**},\dots,s^{**})=\arg\min_{s_i\in S_i}u_j(s_i,s^*,\dots,s^*,s^{**},\dots,s^{**}),\]
\[\arg\max_{s_k\in S_k}u_k(s^*,\dots,s^*,s_k,s^{**},\dots,s^{**})=\arg\min_{s_k\in S_k}u_l(s^*,\dots,s^*,s_k,s^{**},\dots,s^{**}).\]
By the assumption of uniqueness of the best responses, they are unique. By symmetry for each group
\[\arg\max_{s_i\in S_i}u_i(s_i,s^*,\dots,s^*,s^{**},\dots,s^{**})=\arg\min_{s_j\in S_j}u_i(s_j,s^*,\dots,s^*,s^{**},\dots,s^{**}),\]
\[\arg\max_{s_k\in S_k}u_k(s^*,\dots,s^*,s_k,s^{**},\dots,s^{**})=\arg\min_{s_l\in S_l}u_k(s^*,\dots,s^*,s_l,s^{**},\dots,s^{**}).\]
Therefore,
\[u_i(s^*,\dots,s^*,s^{**},\dots,s^{**})=\min_{s_j\in S_j}u_i(s_j,s^*,\dots,s^*,s^{**},\dots,s^{**})\leq u_i(s_j,s^*,\dots,s^*,s^{**},\dots,s^{**}),\]
\[u_k(s^*,\dots,s^*,s^{**},\dots,s^{**})=\min_{s_l\in S_l}u_k(s^*,\dots,s^*,s_l,s^{**},\dots,s^{**})\leq u_k(s^*,\dots,s^*,s_l,s^{**},\dots,s^{**}).\]
We get
\[\max_{s_i\in S_i}u_i(s_i,s^*,\dots,s^*,s^{**},\dots,s^{**})=u_i(s^*,\dots,s^*,s^{**},\dots,s^{**})=\min_{s_j\in S_j}u_i(s_j,s^*,\dots,s^*,s^{**},\dots,s^{**}),\]
\[\max_{s_k\in S_k}u_k(s^*,\dots,s^*,s_k,s^{**},\dots,s^{**})=u_k(s^*,\dots,s^*,s^{**},\dots,s^{**})=\min_{s_l\in S_l}u_k(s^*,\dots,s^*,s_l,s^{**},\dots,s^{**}),\]
They mean
\begin{align}
&\min_{s_j\in S_j}\max_{s_i\in S_i}u_i(s_i,s_j,s^*,\dots,s^*,s^{**},\dots,s^{**})\leq \max_{s_i\in S_i}u_i(s_i,s^*,\dots,s^*,s^{**},\dots,s^{**}) \label{e3} \\
=&\min_{s_j\in S_j}u_i(s_j,s^*,\dots,s^*,s^{**},\dots,s^{**})\leq \max_{s_i\in S_i}\min_{s_j\in S_j}u_i(s_i,s_j,s^*,\dots,s^*,s^{**},\dots,s^{**}).\notag
\end{align}
and
\begin{align}
&\min_{s_l\in S_l}\max_{s_k\in S_k}u_k(s^*,\dots,s^*,s_k,s_l,s^{**},\dots,s^{**})\leq \max_{s_k\in S_k}u_k(s^*,\dots,s^*,s_k,s^{**},\dots,s^{**}) \label{e3a} \\
=&\min_{s_l\in S_l}u_k(s^*,\dots,s^*,s_l,s^{**},\dots,s^{**})\leq \max_{s_k\in S_k}\min_{s_l\in S_l}u_k(s^*,\dots,s^*,s_k,s_l,s^{**},\dots,s^{**}).\notag
\end{align}
On the other hand, since
\[\min_{s_j\in S_j}u_i(s_i,s_j,s^*,\dots,s^*,s^{**},\dots,s^{**})\leq u_i(s_i,s_j,s^*,\dots,s^*,s^{**},\dots,s^{**}),\]
\[\min_{s_k\in S_k}u_k(s^*,\dots,s^*,s_k,s_l,s^{**},\dots,s^{**})\leq u_k(s^*,\dots,s^*,s_k,s_l,s^{**},\dots,s^{**}),\]
we have
\[\max_{s_i\in S_i}\min_{s_j\in S_j}u_i(s_i,s_j,s^*,\dots,s^*,s^{**},\dots,s^{**})\leq \max_{s_i\in S_i}u_i(s_i,s_j,s^*,\dots,s^*,s^{**},\dots,s^{**}),\]
\[\max_{s_k\in S_k}\min_{s_l\in S_l}u_k(s^*,\dots,s^*,s_k,s_l,s^{**},\dots,s^{**})\leq \max_{s_k\in S_k}u_k(s^*,\dots,s^*,s_k,s_l,s^{**},\dots,s^{**}).\]
These inequalities hold for any $s_j$ and $s_l$. Thus,
\[\max_{s_i\in S_i}\min_{s_j\in S_j}u_i(s_i,s_j,s^*,\dots,s^*,s^{**},\dots,s^{**})\leq \min_{s_j\in S_j}\max_{s_i\in S_i}u_i(s_i,s_j,s^*,\dots,s^*,s^{**},\dots,s^{**}),\]
\[\max_{s_k\in S_k}\min_{s_l\in S_l}u_k(s^*,\dots,s^*,s_k,s_l,s^{**},\dots,s^{**})\leq \min_{s_l\in S_l}\max_{s_k\in S_k}u_k(s^*,\dots,s^*,s_k,s_l,s^{**},\dots,s^{**}),\]
With (\ref{e3}) and (\ref{e3a}), we obtain
\begin{equation}
\max_{s_i\in S_i}\min_{s_j\in S_j}u_i(s_i,s_j,s^*,\dots,s^*,s^{**},\dots,s^{**})=\min_{s_j\in S_j}\max_{s_i\in S_i}u_i(s_i,s_j,s^*,\dots,s^*,s^{**},\dots,s^{**}),\label{t1-1}
\end{equation}
\begin{equation}
\max_{s_k\in S_k}\min_{s_l\in S_l}u_k(s^*,\dots,s^*,s_k,s_l,s^{**},\dots,s^{**})=\min_{s_D\in S_D}\max_{s_C\in S_C}u_C(s^*,\dots,s^*,s_k,s_l,s^{**},\dots,s^{**}).\label{t1-1a}
\end{equation}
From
\[\min_{s_j\in S_j}u_i(s_i,s_j,s^*,\dots,s^*,s^{**},\dots,s^{**})\leq u_i(s_i,s^*,\dots,s^*,s^{**},\dots,s^{**}),\]
\[\min_{s_l\in S_l}u_k(s^*,\dots,s^*,s_k,s^{**},\dots,s^{**})\leq u_k(s^*,\dots,s^*,s_k,s^{**},\dots,s^{**}),\]
\[\max_{s_i\in S_i}\min_{s_j\in S_j}u_i(s_i,s_j,s^*,\dots,s^*,s^{**},\dots,s^{**})=\max_{s_i\in S_i}u_i(s_i,s^*,\dots,s^*,s^{**},\dots,s^{**}),\]
and
\[\max_{s_k\in S_k}\min_{s_l\in S_l}u_k(s^*,\dots,s^*,s_k,s_l,s^{**},\dots,s^{**})=\max_{s_k\in S_k}u_k(s^*,\dots,s^*,s_k,s^{**},\dots,s^{**}),\]
we have
\[\arg\max_{s_i\in S_i}\min_{s_j\in S_j}u_i(s_i,s_j,s^*,\dots,s^*,s^{**},\dots,s^{**})=\arg\max_{s_i\in S_i}u_i(s_i,s^*,\dots,s^*,s^{**},\dots,s^{**})=s^*,\]
\[\arg\max_{s_k\in S_k}\min_{s_l\in S_l}u_k(s^*,\dots,s^*,s_k,s_l,s^{**},\dots,s^{**})=\arg\max_{s_k\in S_k}u_k(s^*,\dots,s^*,s_k,s^{**},\dots,s^{**})=s^{**}.\]
From
\[\max_{s_i\in S_i}u_i(s_i,s_j,s^*,\dots,s^*,s^{**},\dots,s^{**})\geq u_i(s_j,s^*,\dots,s^*,s^{**},\dots,s^{**}),\]
\[\max_{s_k\in S_k}u_k(s^*,\dots,s^*,s_k,s_l,s^{**},\dots,s^{**})\geq u_k(s^*,\dots,s^*,s_l,s^{**},\dots,s^{**}),\]
\[\min_{s_j\in S_j}\max_{s_i\in S_i}u_i(s_i,s_j,s^*,\dots,s^*,s^{**},\dots,s^{**})=\min_{s_j\in S_j}u_i(s_j,s^*,\dots,s^*,s^{**},\dots,s^{**}),\]
and
\[\min_{s_l\in S_l}\max_{s_k\in S_k}u_k(s^*,\dots,s^*,s_k,s_l,s^{**},\dots,s^{**})=\min_{s_l\in S_l}u_k(s^*,\dots,s^*,s_l,s^{**},\dots,s^{**}),\]
we get
\[\arg\min_{s_j\in S_j}\max_{s_i\in S_i}u_i(s_i,s_j,s^*,\dots,s^*,s^{**},\dots,s^{**})=\arg\min_{s_j\in S_j}u_i(s_j,s^*,\dots,s^*,s^{**},\dots,s^{**})=s^*,\]
\[\arg\min_{s_l\in S_l}\max_{s_k\in S_k}u_k(s^*,\dots,s^*,s_k,s_l,s^{**},\dots,s^{**})=\arg\min_{s_l\in S_l}u_k(s^*,\dots,s^*,s_l,s^{**},\dots,s^{**})=s^{**}.\]
Therefore,
\begin{align}
&\arg\max_{s_i\in S_i}\min_{s_j\in S_j}u_i(s_i,s_j,s^*,\dots,s^*,s^{**},\dots,s^{**}) \label{t1-2} \\ 
=&\arg\min_{s_j\in S_j}\max_{s_i\in S_i}u_i(s_i,s_j,s^*,\dots,s^*,s^{**},\dots,s^{**})=s^*,\notag
\end{align}
\begin{align}
&\arg\max_{s_k\in S_k}\min_{s_l\in S_l}u_k(s^*,\dots,s^*,s_k,s_l,s^{**},\dots,s^{**}) \label{t1-2a} \\
=&\arg\min_{s_l\in S_l}\max_{s_k\in S_k}u_k(s^*,\dots,s^*,s_k,s_l,s^{**},\dots,s^{**})=s^{**}.\notag
\end{align}
\end{proof}

Next we show the following theorem. 
\begin{theorem}
Sion's minimax theorem with the coincidence of the maximin strategy and the minimax strategy implies the existence of a Nash equilibrium which is symmetric in each group.
\end{theorem}
\begin{proof}
We denote a state such that Players $1, 2, \dots, m$ choose $\tilde{s}$, and Players $m+1, m+2, \dots, n$ choose $\hat{s}$ by $(\tilde{s},\dots,\tilde{s},\hat{s}\dots,\hat{s})$.

Let $\tilde{s}$ and $\hat{s}$ be the values of $s_i$'s for $i\in \gra$ and $s_k$'s for $k\in \grb$ such that
\[\tilde{s}=\arg\max_{s_i\in S_i}\min_{s_j\in S_j}u_i(s_i,s_j,\tilde{s},\dots,\tilde{s},\hat{s},\dots,\hat{s})=\arg\min_{s_j\in S_j}\max_{s_i\in S_i}u_i(s_i,s_j,\tilde{s},\dots,\tilde{s},\hat{s},\dots,\hat{s}),\]
\[\hat{s}=\arg\max_{s_k\in S_k}\min_{s_l\in S_l}u_k(\tilde{s},\dots,\tilde{s},s_k, s_l,\hat{s},\dots,\hat{s})=\arg\min_{s_l\in S_l}\max_{s_k\in S_k}u_k(\tilde{s},\dots,\tilde{s},s_k, s_l,\hat{s},\dots,\hat{s}),\]
\begin{align*}
&\max_{s_i\in S_i}\min_{s_j\in S_j}u_i(s_i,s_j,\tilde{s},\dots,\tilde{s},\hat{s},\dots,\hat{s})=\min_{s_j\in S_j}u_i(s_j,\tilde{s},\dots,\tilde{s},\hat{s},\dots,\hat{s})\\
&=\min_{s_j\in S_j}\max_{s_i\in S_i}u_i(s_i,s_j,\tilde{s},\dots,\tilde{s},\hat{s},\dots,\hat{s})=\max_{s_i\in S_i}u_i(s_i,\tilde{s},\dots,\tilde{s},\hat{s},\dots,\hat{s}),
\end{align*}
and
\begin{align*}
&\max_{s_k\in S_k}\min_{s_l\in S_l}u_k(\tilde{s},\dots,\tilde{s},s_k,s_l,\hat{s},\dots,\hat{s})=\min_{s_l\in S_l}u_k(\tilde{s},\dots,\tilde{s},s_l,\hat{s},\dots,\hat{s})\\
&=\min_{s_l\in S_l}\max_{s_k\in S_k}u_k(\tilde{s},\dots,\tilde{s},s_k, s_l,\hat{s},\dots,\hat{s})=\max_{s_k\in S_k}u_k(\tilde{s},\dots,\tilde{s},s_k,\hat{s},\dots,\hat{s}).
\end{align*}
Since
\[u_i(s_j,\tilde{s},\dots,\tilde{s},\hat{s},\dots,\hat{s})\leq \max_{s_i\in S_i}u_i(s_i,s_j,\tilde{s},\dots,\tilde{s},\hat{s},\dots,\hat{s}),\]
\[\min_{s_j\in S_j}u_i(s_j,\tilde{s},\dots,\tilde{s},\hat{s},\dots,\hat{s})=\min_{s_j\in S_j}\max_{s_i\in S_i}u_i(s_i,s_j,\tilde{s},\dots,\tilde{s},\hat{s},\dots,\hat{s}),\]
we get
\[\arg\min_{s_j\in S_j}u_i(s_j,\tilde{s},\dots,\tilde{s},\hat{s},\dots,\hat{s})=\arg\min_{s_j\in S_j}\max_{s_i\in S_i}u_i(s_i,s_j,\tilde{s},\dots,\tilde{s},\hat{s},\dots,\hat{s})=\tilde{s}.\]
Similarly, from
\[u_k(\tilde{s},\dots,\tilde{s},s_l,\hat{s},\dots,\hat{s})\leq \max_{s_k\in S_k}u_k(\tilde{s},\dots,\tilde{s},s_k, s_l,\hat{s},\dots,\hat{s}),\]
\[\min_{s_l\in S_l}u_k(\tilde{s},\dots,\tilde{s},s_l,\hat{s},\dots,\hat{s})=\min_{s_l\in S_l}\max_{s_k\in S_k}u_k(\tilde{s},\dots,\tilde{s},s_k,s_l,\hat{s},\dots,\hat{s}),\]
we get
\[\arg\min_{s_l\in S_l}u_k(\tilde{s},\dots,\tilde{s},s_l,\hat{s},\dots,\hat{s})=\arg\min_{s_l\in S_l}\max_{s_k\in S_k}u_k(\tilde{s},\dots,\tilde{s},s_k, s_l,\hat{s},\dots,\hat{s})=\hat{s}.\]
Since
\[u_i(s_i,\tilde{s},\dots,\tilde{s},\hat{s},\dots,\hat{s})\geq \min_{s_j\in S_j}u_i(s_i,s_j,\tilde{s},\dots,\tilde{s},\hat{s},\dots,\hat{s}),\]
and
\[\max_{s_i\in S_i}u_i(s_i,\tilde{s},\dots,\tilde{s},\hat{s},\dots,\hat{s})=\max_{s_i\in S_i}\min_{s_j\in S_j}u_i(s_i,s_j,\tilde{s},\dots,\tilde{s},\hat{s},\dots,\hat{s}),\]
we obtain
\[\arg\max_{s_i\in S_i}u_i(s_i,\tilde{s},\dots,\tilde{s},\hat{s},\dots,\hat{s})=\arg\max_{s_i\in S_i}\min_{s_j\in S_j}u_i(s_i,s_j,\tilde{s},\dots,\tilde{s},\hat{s},\dots,\hat{s})=\tilde{s}.\]
Similarly, from
\[u_k(\tilde{s},\dots,\tilde{s},s_k,\hat{s},\dots,\hat{s})\geq \min_{s_l\in S_l}u_k(\tilde{s},\dots,\tilde{s},s_k, s_l,\hat{s},\dots,\hat{s}),\]
and
\[\max_{s_k\in S_k}u_k(\tilde{s},\dots,\tilde{s},s_k,\hat{s},\dots,\hat{s})=\max_{s_k\in S_k}\min_{s_l\in S_l}u_k(\tilde{s},\dots,\tilde{s},s_l,\hat{s},\dots,\hat{s}),\]
we obtain
\[\arg\max_{s_k\in S_k}u_k(\tilde{s},\dots,\tilde{s},s_k,\hat{s},\dots,\hat{s})=\arg\max_{s_k\in S_k}\min_{s_l\in S_l}u_k(\tilde{s},\dots,\tilde{s},s_k, s_l,\hat{s},\dots,\hat{s})=\hat{s}.\]
Therefore,
\[u_i(s_j,\tilde{s},\dots,\tilde{s},\hat{s},\dots,\hat{s})\geq u_i(\tilde{s},\dots,\tilde{s},\hat{s},\dots,\hat{s})\geq u_i(s_i,\tilde{s},\dots,\tilde{s},\hat{s},\dots,\hat{s}),\]
\[u_k(\tilde{s},\dots,\tilde{s},s_l,\hat{s},\dots,\hat{s})\geq u_k(\tilde{s},\dots,\tilde{s},\hat{s},\dots,\hat{s})\geq u_k(\tilde{s},\dots,\tilde{s},s_k,\hat{s},\dots,\hat{s}).\]

Thus, $(\tilde{s},\dots,\tilde{s},\hat{s}\dots,\hat{s})$ is a Nash equilibrium which is symmetric in each group. 
\end{proof}

\section{Example of relative profit maximization in each group of six-firms oligopoly}\label{ex}

Consider a six-players game. The players are A, B, C, D, E and F. Suppose that the payoff functions of Players A, B and E are symmetric, and those of Players C, D and F are symmetric. The payoff functions of the players are
\begin{align*}
\pi_A=&(a-x_A-x_B-x_E-bx_C-bx_D-bx_F)x_A-c_Ax_A\\
&-\frac{1}{2}[(a-x_A-x_B-x_E-bx_C-bx_D-bx_F)x_B-c_Ax_B\\
&+(a-x_A-x_B-x_E-bx_C-bx_D-bx_F)x_E-c_Ax_E],
\end{align*}
\begin{align*}
\pi_B=&(a-x_A-x_B-x_E-bx_C-bx_D-bx_F)x_B-c_Ax_B\\
&-\frac{1}{2}[(a-x_A-x_B-x_E-bx_C-bx_D-bx_F)x_A-c_Ax_A\\
&+(a-x_A-x_B-x_E-bx_C-bx_D-bx_F)x_E-c_Ax_E],
\end{align*}
\begin{align*}
\pi_E=&(a-x_A-x_B-x_E-bx_C-bx_D-bx_F)x_B-c_Ax_B\\
&-\frac{1}{2}[(a-x_A-x_B-x_E-bx_C-bx_D-bx_F)x_B-c_Ax_B\\
&+(a-x_A-x_B-x_E-bx_C-bx_D)x_B-c_Ax_B],
\end{align*}
\begin{align*}
\pi_C=&(a-x_C-x_D-x_F-bx_E-bx_A-bx_B)x_C-c_Cx_C\\
&-\frac{1}{2}[(a-x_C-x_D-x_F-bx_E-bx_A-bx_B)x_D-c_Cx_D\\
&+(a-x_C-x_D-x_F-bx_E-bx_A-bx_B)x_F-c_Cx_F],
\end{align*}
\begin{align*}
\pi_D=&(a-x_C-x_D-x_F-bx_E-bx_A-bx_B)x_D-c_Cx_D\\
&-\frac{1}{2}[(a-x_C-x_D-x_F-bx_E-bx_A-bx_B)x_C-c_Cx_C\\
&+(a-x_C-x_D-x_F-bx_E-bx_A-bx_B)x_F-c_Cx_F],
\end{align*}
\begin{align*}
\pi_F=&(a-x_C-x_D-x_F-bx_E-bx_A-bx_B)x_F-c_Cx_F\\
&-\frac{1}{2}[(a-x_C-x_D-x_F-bx_E-bx_A-bx_B)x_C-c_Cx_C\\
&+(a-x_C-x_D-x_F-bx_E-bx_A-bx_B)x_D-c_Cx_D].
\end{align*}
This is a model of relative profit maximization in each group in a six firms oligopoly with two groups. $x_A$, $x_B$, $x_C$, $x_D$, $x_E$ and $x_F$ are the outputs of the firms, and $p_A$, $p_B$, $p_C$, $p_D$, $p_E$ and $p_F$ are the prices of their goods. The demand functions are symmetric for Firms A, B and E, and they have the same cost functions, also the demand functions are symmetric for Firms C, D and F, and they have the same cost functions. However, the demand function for Firm A (or B or E) is not symmetric for Firm C (or D or F), and the demand function for Firm C (or D or F) is not symmetric for Firm A (or B or E). Firm A's (or Firm B's or Firm E's) cost function is different from the cost function of Firm C (or Firm D or Firm F). The cost functions of the firms are linear and there is no fixed cost.

We assume that Firm A (or B or E) maximizes its profit relatively to the profit of Firm B and E (or A and E, or B and E), and  Firm C (or D or F) maximizes its profit relatively to the profit of Firm D and F (or C and F, or D and F). Note that
\[\pi_A+\pi_B+\pi_E=0,\ \pi_C+\pi_D+\pi_F=0.\]
Thus, this is a model of zero-sum game in each group with two groups.

Under the assumption of Cournot type behavior, the equilibrium outputs are
\[x_A=\frac{bc_C-c_A-ab+a}{3(1-b)(1+b)},\]
\[x_B=\frac{bc_C-c_A-ab+a}{3(1-b)(1+b)},\]
\[x_C=\frac{bc_A-c_C-ab+a}{3(1-b)(1+b)},\]
\[x_D=\frac{bc_A-c_C-ab+a}{3(1-b)(1+b)},\]
\[x_E=\frac{bc_C-c_A-ab+a}{3(1-b)(1+b)},\]
\[x_F=\frac{bc_A-c_C-ab+a}{3(1-b)(1+b)}.\]
The equilibrium prices of the goods are
\[p_A=c_A,\]
\[p_B=c_A,\]
\[p_C=c_C,\]
\[p_D=c_C,\]
\[p_E=c_A,\]
\[p_C=c_C.\]

Therefore, the prices of the goods are equal to the marginal costs in each group.

The maximin and minimax strategies between Firms A and B are
\[\arg\max_{x_A}\min_{x_B}\pi_A,\ \arg\min_{x_B}\max_{x_A}\pi_A.\]
Similarly, we can define the maximin and minimax strategies between Firms A and E, those between Firms B and E, Firms E and A, Firms E and B.

Those between Firm C and D are
\[\arg\max_{x_C}\min_{x_D}\pi_C,\ \arg\min_{x_D}\max_{x_C}\pi_C.\]
Similarly, we can define the maximin and minimax strategies between Firms C and F, those between Firms D and F, Firms F and C, Firms F and D.

In our example, under the assumption that $x_E=x_A$, we obtain
\[\arg\max_{x_A}\min_{x_B}\pi_A=\frac{bc_C-c_A-ab+a}{3(1-b)(1+b)},\]
\[\arg\min_{x_B}\max_{x_A}\pi_A=\frac{bc_C-c_A-ab+a}{3(1-b)(1+b)},\]
\[\arg\max_{x_C}\min_{x_D}\pi_C=\frac{bc_A-c_C-ab+a}{3(1-b)(1+b)},\]
\[\arg\min_{x_D}\max_{x_C}\pi_C=\frac{bc_A-c_C-ab+a}{3(1-b)(1+b)},\]
\[\arg\max_{x_E}\max_{x_A}\pi_E=\frac{bc_C-c_A-ab+a}{3(1-b)(1+b)},\]
\[\arg\min_{x_F}\max_{x_C}\pi_C=\frac{bc_A-c_C-ab+a}{3(1-b)(1+b)},\]
and so on. They are the same as Nash equilibrium strategies.

\section{Concluding Remark}

We think that our analysis can be easily extended to a case with more than two groups.

\appendix

\section{Note on the case where Assumption \ref{as1} is not assumed.}

Let $(\tilde{s},\hat{s},s^1,s^2)$ be the solution (fixed point) of the following equations.

\begin{equation*}
\tilde{s}=\arg\max_{s_i\in S_i}\min_{s_j\in S_j}u_i(s_i,s_j,\tilde{s},\dots,\tilde{s},s^2,\hat{s},\dots,\hat{s}),
\end{equation*}
\[s^1=\arg\min_{s_j\in S_j}\max_{s_i\in S_i}u_i(s_i,s_j,\tilde{s},\dots,\tilde{s},s^2,\hat{s},\dots,\hat{s}),\]
\begin{equation*}
\hat{s}=\arg\max_{s_k\in S_k}\min_{s_l\in S_l}u_k(s^1,\tilde{s},\dots,\tilde{s},s_k,s_l,\hat{s},\dots,\hat{s}).
\end{equation*}
and
\[s^2=\arg\min_{s_l\in S_l}\max_{s_k\in S_k}u_k(s^1,\tilde{s},\dots,\tilde{s},s_k,s_l,\hat{s},\dots,\hat{s}),\]
with $s_j=s^1$ and $s_l=s^2$. By (\ref{as0}) and (\ref{as0a})
\begin{align*}
&\max_{s_i\in S_i}\min_{s_j\in S_j}u_i(s_i, s_j,\tilde{s},\dots,\tilde{s},s^2,\hat{s},\dots,\hat{s})=\min_{s_j\in S_j}u_i(s_j,\tilde{s},\dots,\tilde{s},s^2,\hat{s},\dots,\hat{s})\\
=&\min_{s_j\in S_j}\max_{s_i\in S_i}u_i(s_i, s_j,\tilde{s},\dots,\tilde{s},s^2,\hat{s},\dots,\hat{s})=\max_{s_i\in S_i}u_i(s_i, s^1,\tilde{s},\dots,\tilde{s},s^2,\hat{s},\dots,\hat{s}),
\end{align*}
and
\begin{align*}
&\max_{s_k\in S_k}\min_{s_l\in S_l}u_k(s^1,\tilde{s},\dots,\tilde{s},s_k,s_l,\hat{s},\dots,\hat{s})=\min_{s_l\in S_l}u_k(s^1,\tilde{s},\dots,\tilde{s},s_l,\hat{s},\dots,\hat{s})\\
=&\min_{s_l\in S_l}\max_{s_k\in S_k}u_k(s^1,\tilde{s},\dots,\tilde{s},s_k,s_l,\hat{s},\dots,\hat{s})=\max_{s_k\in S_k}u_k(s^1,\tilde{s},\dots,\tilde{s},s_k,s^2,\hat{s},\dots,\hat{s}).
\end{align*}
Since
\[\max_{s_i\in S_i}u_i(s_i,s_j,\tilde{s},\dots,\tilde{s},s^2,\hat{s},\dots,\hat{s})\geq u_i(s_j,\tilde{s},\dots,\tilde{s},s^2,\hat{s},\dots,\hat{s}),\]
\[\min_{s_j\in S_j}\max_{s_i\in S_i}u_i(s_i, s_j,\tilde{s},\dots,\tilde{s},s^2,\hat{s},\dots,\hat{s})=\min_{s_j\in S_j}u_i(s_j,\tilde{s},\dots,\tilde{s},s^2,\hat{s},\dots,\hat{s}),\]
we have
\begin{equation*}
\arg\min_{s_j\in S_j}\max_{s_i\in S_i}u_i(s_i, s_j,\tilde{s},\dots,\tilde{s},s^2,\hat{s},\dots,\hat{s})=\arg\min_{s_j\in S_j}u_i(s_j,\tilde{s},\dots,\tilde{s},s^2,\hat{s},\dots,\hat{s})=s^1.
\end{equation*}
Similarly, from
\[\max_{s_k\in S_k}u_k(s^1,\tilde{s},\dots,\tilde{s},s_k,s_l,\hat{s},\dots,\hat{s})\geq u_k(s^1,\tilde{s},\dots,\tilde{s},s_l,\hat{s},\dots,\hat{s}),\]
\[\min_{s_l\in S_l}\max_{s_k\in S_k}u_k(s^1,\tilde{s},\dots,\tilde{s},s_k,s_l,\hat{s},\dots,\hat{s})=\min_{s_l\in S_l}u_i(s^1,\tilde{s},\dots,\tilde{s},s_l,\hat{s},\dots,\hat{s}),\]
we have
\begin{equation*}
\arg\min_{s_l\in S_l}\max_{s_k\in S_k}u_k(s^1,\tilde{s},\dots,\tilde{s},s_k,s_l,\hat{s},\dots,\hat{s})=\arg\min_{s_l\in S_l}u_k(s^1,\tilde{s},\dots,\tilde{s},s_l,\hat{s},\dots,\hat{s})=s^2.
\end{equation*}
Since
\[\min_{s_j\in S_j}u_i(s_i, s_j,\tilde{s},\dots,\tilde{s},s^2,\hat{s},\dots,\hat{s})\leq u_i(s_i, s^1,\tilde{s},\dots,\tilde{s},s^2,\hat{s},\dots,\hat{s}),\]
\[\min_{s_l\in S_l}u_k(s^1,\tilde{s},\dots,\tilde{s},s_k,s_l,\hat{s},\dots,\hat{s})\leq u_k(s^1,\tilde{s},\dots,\tilde{s},s_k,s^2,\hat{s},\dots,\hat{s}),\]
\[\max_{s_i\in S_i}\min_{s_j\in S_j}u_i(s_i,s_j,\tilde{s},\dots,\tilde{s},s^2,\hat{s},\dots,\hat{s})=\max_{s_i\in S_i}u_i(s_i, s^1,\tilde{s},\dots,\tilde{s},s^2,\hat{s},\dots,\hat{s}),\]
and
\[\max_{s_k\in S_k}\min_{s_l\in S_l}u_k(s^1,\tilde{s},\dots,\tilde{s},s_k,s_l,\hat{s},\dots,\hat{s})=\max_{s_k\in S_k}u_k(s^1,\tilde{s},\dots,\tilde{s},s_k,s^2,\hat{s},\dots,\hat{s}),\]
we have
\begin{equation}
\arg\max_{s_i\in S_i}\min_{s_j\in S_j}u_i(s_i, s_j,\tilde{s},\dots,\tilde{s},s^2,\hat{s},\dots,\hat{s})=\arg\max_{s_i\in S_i}u_i(s_i, s^1, \tilde{s},\dots,\tilde{s},s^2,\hat{s},\dots,\hat{s})=\tilde{s},\label{e5}
\end{equation}
\begin{equation}
\arg\max_{s_k\in S_k}\min_{s_l\in S_l}u_k(s^1,\tilde{s},\dots,\tilde{s},s_k,s_l,\hat{s},\dots,\hat{s})=\arg\max_{s_k\in S_k}u_k(s^1,\tilde{s},\dots,\tilde{s},s_k,s^2,\hat{s},\dots,\hat{s})=\hat{s}.\label{e5a}
\end{equation}
Because the game is zero-sum in each group,
\[\sum_{i=1, i\neq j}^mu_i(s_j,\tilde{s},\dots,\tilde{s},s^2,\hat{s},\dots,\hat{s})+u_j(s_j,\tilde{s},\dots,\tilde{s},s^2,\hat{s},\dots,\hat{s})=0,\]
\[\sum_{k=m+1, k\neq l}^nu_k(s^1,\tilde{s},\dots,\tilde{s},s_l,\hat{s},\dots,\hat{s}+u_l(s^1,\tilde{s},\dots,\tilde{s},s_l,\hat{s},\dots,\hat{s})=0.\]
By symmetry for each group
\[(m-1)u_i(s_j,\tilde{s},\dots,\tilde{s},s^2,\hat{s},\dots,\hat{s})+u_j(s_j,\tilde{s},\dots,\tilde{s},s^2,\hat{s},\dots,\hat{s})=0,\]
\[(n-m-1)u_k(s^1,\tilde{s},\dots,\tilde{s},s_l,\hat{s},\dots,\hat{s}+u_l(s^1,\tilde{s},\dots,\tilde{s},s_l,\hat{s},\dots,\hat{s})=0.\]
Thus,
\[(m-1)u_i(s_j,\tilde{s},\dots,\tilde{s},s^2,\hat{s},\dots,\hat{s})=-u_j(s_j,\tilde{s},\dots,\tilde{s},s^2,\hat{s},\dots,\hat{s}),\]
\[(n-m-1)u_k(s^1,\tilde{s},\dots,\tilde{s},s_l,\hat{s},\dots,\hat{s})=-u_l(s^1,\tilde{s},\dots,\tilde{s},s_l,\hat{s},\dots,\hat{s}).\]
They mean
\begin{equation}
\arg\min_{s_j\in S_j}u_i(s_j,\tilde{s},\dots,\tilde{s},s^2,\hat{s},\dots,\hat{s})=\arg\max_{s_j\in S_j}u_j(s_j,\tilde{s},\dots,\tilde{s},s^2,\hat{s},\dots,\hat{s})=s^1,\label{e4}
\end{equation}
\begin{equation}
\arg\min_{s_l\in S_l}u_k(s^1,\tilde{s},\dots,\tilde{s},s_l,\hat{s},\dots,\hat{s})=\arg\max_{s_l\in S_l}u_l(s^1,\tilde{s},\dots,\tilde{s},s_l,\hat{s},\dots,\hat{s})=s^2.\label{e4a}
\end{equation}

Therefore, if $s^1\neq \tilde{s}$ or $s^2\neq \hat{s}$, there may exist a Nash equilibrium denoted as follows;
\[(\tilde{s},\dots, s^1,\dots, \tilde{s},\hat{s},\dots,s^2,\dots \hat{s}),\]
We may have $s^1=\tilde{s}$ or $s^2=\hat{s}$.

\end{document}